\documentclass[12pt]{amsart}
\usepackage{graphicx}
\usepackage{amsmath,amssymb}
\newtheorem{thm}{Theorem}[section]
\newtheorem{cor}[thm]{Corollary}
\newtheorem{lem}[thm]{Lemma}
\newtheorem{prop}[thm]{Proposition}
\newtheorem{rmk}[thm]{Remark}

\begin{document}
\title{On semisimple Hopf algebras of dimension $2q^3$}
\author[J. Dong, S. Wang]{Jingcheng Dong$^{a,b}$, Shuanhong Wang$^a$}

\address[a]{Department of mathematics, Southeast University, Nanjing 210096, Jiangsu, People's Republic of
China}
\address[b]{College of Engineering, Nanjing Agricultural University, Nanjing
210031, Jiangsu, People's Republic of China}

\email[J. Dong]{dongjc@njau.edu.cn}

\email[S. Wang]{shuanhwang2002@yahoo.com}

\begin{abstract}
Let $q$ be a prime number, $k$ an algebraically closed field of
characteristic $0$, and $H$ a semisimple Hopf algebra of dimension
$2q^3$. This paper proves that $H$ is always semisolvable. That is,
such Hopf algebras can be obtained by (a number of) extensions from
group algebras or duals of group algebras.
\end{abstract}

 \maketitle

\section{Introduction}\label{sec1}
The notions of upper and lower semisolvability for
finite-dimensional Hopf algebras were introduced by Montgomery and
Whiterspoon \cite{Montgomery}, as generalizations of the notion of
solvability for finite groups. In particular, if a
finite-dimensional Hopf algebra $A$ is semisolvable then $A$ can be
obtained by a number of extensions from group algebras or duals of
group algebras. Therefore, in analogy with the situations for finite
groups, it is enough for many applications to know that a Hopf
algebra is semisolvable.

The known examples of semisimple Hopf algebras which are
semisolvable are those of dimension $p^n,pq^2,pqr$ and those of
dimension less than $60$ (except $36$), where $p,q,r$ are distinct
prime numbers and $n$ is a natural number. See
\cite{Montgomery,Etingof2,Natale4} for details. Our present work is
devoted to providing a new class of semisimple Hopf algebras which
are semisolvable. It can also be viewed as a generalization of
\cite[Chapter 12]{Natale4}.

The paper is organized as follows. In Section \ref{sec2}, we recall
the definitions and some of the basic properties of semisolvability,
characters, Radford biproducts and Drinfeld double, respectively.
Some useful lemmas are also contained in this section.

In Section \ref{sec3}, we present our main results. Let $G(H)$
denote the group of group-like elements in a semisimple Hopf algebra
$H$. By examining every possible order of $G(H)$, we prove that if
the dimension of $H$ is $2q^3$ then $H$ is semisolvable, where $q$
is a prime number.

Throughout this paper, all modules and comodules are left modules
and left comodules, and moreover they are finite-dimensional over an
algebraically closed field $k$ of characteristic $0$. $\otimes$,
${\rm dim}$ mean $\otimes _k$, ${\rm dim}_k$, respectively. If $G$
is a finite group, $kG$ denotes the group algebra of $G$, and $k^G$
means $(kG)^*$. If $g\in G$ then $\langle g\rangle$ denotes the
subgroup of $G$ generated by $g$. Further $C_n$ denotes the cyclic
group of order $n$.  For two positive integers $m$ and $n$,
$gcd(m,n)$ denotes the greatest common divisor of $m,n$. Our
references for the theory of Hopf algebras are \cite{Montgomery2} or
\cite{Sweedler}.

\section{Preliminaries}\label{sec2}

\subsection{Characters}Throughout this subsection, $H$ will be a semisimple Hopf
algebra over $k$. The main result in \cite{Larson} states that $H$
is also cosemisimple.

We next recall some of the terminology and conventions from
\cite{Nichols} that will be used throughout this paper.

Let $V$ be an $H$-comodule. The character of $V$ is the element
$\chi=\chi_V\in H$ defined by $\langle f,\chi\rangle={\rm Tr}_V(f)$
for all $f\in H^*$. The degree of $\chi$ is defined to be the
integer ${\rm deg}\chi=\varepsilon(\chi)={\rm dim}V$. We shall use
$X_t$ to denote the set of all irreducible characters of $H$ of
degree $t$.

All irreducible characters of $H$ span a subalgebra $R(H^*)$ of $H$,
which is called the character algebra of $H^*$. The antipode $S$
induces an anti-algebra involution $*: R(H^*)\to R(H^*)$, given by
$\chi\mapsto\chi^*:=S(\chi)$.

Let $\chi_U,\chi_V\in R(H^*)$ be the characters of the $H$-comodules
$U$ and $V$, respectively. The integer $m(\chi_U,\chi_V)={\rm
dimHom}^H(U,V)$ is defined to be the multiplicity of $U$ in $V$. Let
$\widehat{H}$ denote the set of irreducible characters of $H$. Then
$\widehat{H}$ is a basis of $R(H^*)$. If $\chi\in R(H^*)$, we may
write $\chi=\sum_{\alpha\in \widehat{H}}m(\alpha,\chi)\alpha$.

For any group-like element $g$ in $G(H)$, $m(g,\chi\chi^{*})>0$ if
and only if $m(g,\chi\chi^{*})= 1$ if and only if $g\chi=\chi$ for
every $\chi \in \widehat{H}$. The set of such group-like elements
forms a subgroup of $G(H)$, of order at most $({\rm deg}\chi)^2$.
See \cite[Theorem 10]{Nichols}. Denote this subgroup by $G[\chi]$.
In particular, we have
$$\chi\chi^*=\sum_{g\in G[\chi]}g+\sum_{\alpha\in \widehat{H},{\rm
deg}\alpha>1}m(\alpha,\chi\chi^*)\alpha.\eqno(2.1)$$

A subalgebra $A$ of $R(H^*)$ is called a standard subalgebra if $A$
is spanned by irreducible characters of $H$. Let $X$ be a subset of
$\widehat{H}$. Then $X$ spans a standard subalgebra of $R(H^*)$ if
and only if the product of characters in $X$ decomposes as a sum of
characters in $X$. There is a bijection between $*$-invariant
standard subalgebras of $R(H^*)$ and Hopf subalgebras of $H$. See
\cite[Theorem 6]{Nichols}.

$H$ is said to be of type $(d_1,n_1;\cdots;d_s,n_s)$ as a coalgebra
if $d_1=1,d_2,\cdots,d_s$ are the dimensions of the irreducible
$H$-comodules and  $n_i$ is the number of the non-isomorphic
irreducible $H$-comodules of dimension $d_i$. That is, as a
coalgebra, $H$ is isomorphic to a direct sum of full matrix
coalgebras
$$H\cong k^{(n_1)}\oplus \bigoplus_{i=2}^{s}M_{d_i}(k)^{(n_i)}.\eqno(2.2)$$

If $H^*$ is of type $(d_1,n_1;\cdots;d_s,n_s)$ as a coalgebra, then
$H$ is said to be of type $(d_1,n_1;\cdots;d_s,n_s)$ as an algebra.

\begin{lem}\label{lem1}
Let $\chi$ be an irreducible character of $H$. Then

 (1)\,The order of $G[\chi]$ divides $({\rm deg}\chi)^2$.

 (2)\,The order of $G(H)$ divides $n({\rm deg}\chi)^2$, where $n$ is the
 number of non-isomorphic irreducible characters of degree ${\rm deg}\chi$.
\end{lem}
\begin{proof} It follows from Nichols-Zoeller Theorem \cite{Nichols2}. See
also \cite[Lemma 2.2.2]{Natale1}.
\end{proof}

\subsection{Semisolvability}\label{sec2-2}

Let $B$ be a finite-dimensional Hopf algebra over $k$. A Hopf
subalgebra $A\subseteq B$ is called normal if $h_1AS(h_2)\subseteq
A$ or $S(h_1)Ah_2\subseteq A$, for all $h\in B$. If $B$ does not
contain proper normal Hopf subalgebras then it is called simple.
Dualizing the notion of normal Hopf subalgebra, we obtain the notion
of conormal quotient Hopf algebra. The notion of simplicity is
self-dual, that is, $B$ is simple if and only if $B^*$ is simple.

Let $K\subseteq A$ be a normal Hopf subalgebra. Then $B = A/AK^+$ is
a conormal quotient Hopf algebra and the sequence of Hopf algebra
maps $k\to K\to A\to B\to k$ is an exact sequence of Hopf algebras.
In this case we shall say that $A$ is an extension of $B$ by $K$.

The extension above is called abelian if $K$ is commutative and $B$
is cocommutative. In this case $K\cong  k^N$ and $B\cong kF$, for
some finite groups $N$ and $F$.

The following lemma is a direct consequence of \cite[Corollary
1.4.3]{Natale4}.

\begin{lem}\label{lem2}Let $\pi: A\to B$ be a conormal quotient Hopf algebra. Suppose that
${\rm dim}B$ is the least prime number dividing ${\rm dim}A$. Then
$G(B^*)\subseteq Z(A^*)\cap G(A^*)$.
\end{lem}

Let $\pi:H\to B$ be a Hopf algebra map and consider the subspaces of
coinvariants
$$H^{co\pi}=\{h\in H|(id\otimes \pi)\Delta(h)=h\otimes 1\}, \mbox{and\,}$$
$$^{co\pi}\!H=\{h\in H|(\pi\otimes id)\Delta(h)=1\otimes h\}.$$
Then $H^{co\pi}$ (respectively, $^{co\pi}H$) is a left
(respectively, right) coideal subalgebra of $H$. Moreover, we have
$${\rm dim}H ={\rm dim}H^{co\pi}{\rm dim}\pi(H) ={\rm dim}{}^{co\pi}H{\rm dim}\pi(H).$$

The left coideal subalgebra $H^{co\pi}$ is stable under the left
adjoint action of $H$. Moreover $H^{co\pi} ={}^{co\pi}H$ if and only
if $H^{co\pi}$ is a (normal) Hopf subalgebra of $H$. See
\cite{Schneider} for more details.

The following lemma is taken from \cite[Section 1.3]{Natale4}.
\begin{lem}\label{lem3}
Let $\pi:H\to B$ be a Hopf surjection and $A$ a Hopf subalgebra of
$H$ such that $A\subseteq H^{co\pi}$. Then ${\rm dim}A$ divides
${\rm dim}H^{co\pi}$.
\end{lem}

By definition, $H$ is called lower semisolvable if there exists a
chain of Hopf subalgebras
$$H_{n+1} = k\subseteq
H_{n}\subseteq\cdots \subseteq H_1 = H$$ such that $H_{i+1}$ is a
normal Hopf subalgebra of $H_i$, for all $i$, and all quotients
$H_{i}/H_{i}H^+_{i+1}$ are trivial. That is, they are isomorphic to
a group algebra or a dual group algebra. Dually, $H$ is called upper
semisolvable if there exists a chain of quotient Hopf algebras
$$H_{(0)} =
H\xrightarrow{\pi_1}H_{(1)}\xrightarrow{\pi_2}\cdots\xrightarrow{\pi_n}H(n)
= k$$ such that $H_{(i-1)}^{co\pi_{i}}$ is a normal Hopf subalgebra
of $H_{(i-1)}$, and all $H_{(i-1)}^{co\pi_i}$ are trivial.

By \cite[Corollary 3.3]{Montgomery}, we have that $H$ is upper
semisolvable if and only if $H^*$ is lower semisolvable. $H$ is
called semisolvable if it is upper semisolvable or lower
semisolvable.

\begin{prop}\label{prop1}
Let $H$ be a semisimple Hopf algebra of dimension $pq^3$, where
$p,q$ are distinct prime numbers. If $H$ is not simple as a Hopf
algebra then it is semisolvable.
\end{prop}
\begin{proof}
By assumption, $H$ has a proper normal Hopf subalgebra $K$.
Moreover, by Nichols-Zoeller Theorem \cite{Nichols2}, ${\rm dim}K$
divides ${\rm dim}H=pq^3$. We shall examine every possible ${\rm
dim}K$.

If ${\rm dim}K=q^2$ or $pq$ then $k\subseteq K\subseteq H$ is a
chain such that $K$ and $H/HK^+$ are both trivial (see
\cite{Etingof,Masuoka1}). Hence, $H$ is lower semisolvable.

If ${\rm dim}K=q^3$ then \cite{Masuoka1} shows that $K$ has a
non-trivial central group-like element $g$. Let $L=k\langle
g\rangle$ be the group algebra of the cyclic group $\langle
g\rangle$ generated by $g$. Then $k\subseteq L\subseteq K\subseteq
H$ is a chain such that $L,K/KL^+$ and $H/HK^+$ are all trivial (see
\cite{Zhu}). Hence, $H$ is lower semisolvable.

If ${\rm dim}K=pq^2$ then  \cite[Proposition 9.9]{Etingof2} and
\cite[Theorem 5.4.1]{Natale3} show that $K$ has a proper normal Hopf
subalgebra $L$ of dimension $p,q,pq$ or $q^2$. Then $k\subseteq
L\subseteq K\subseteq H$ is a chain such that $L,K/KL^+$ and
$H/HK^+$ are all trivial. Hence, $H$ is lower semisolvable.

Finally, we consider the case that ${\rm dim}K=p$ or $q$. Let $L$ be
a proper normal Hopf subalgebra of $H/HK^+$ (Notice that $H/HK^+$ is
not simple). Write $\overline{K}=H/HK^+$ and
$\overline{L}=\overline{K}/\overline{K}L^+$.  Then
$H\xrightarrow{\pi_1}\overline{K}\xrightarrow{\pi_2}\overline{L}\xrightarrow{}
k$ is a chain such that every map is normal and $H^{co\pi_1}$,
$(\overline{K})^{co\pi_2}$ are trivial.  Hence, $H$ is upper
semisolvable.
\end{proof}

\begin{rmk}
Let $H$ be a semisimple Hopf algebra of dimension $p^2q^2$, where
$p,q$ are prime numbers. If $H$ is not simple then $H$ is also
semisolvable. Indeed, since every  (quotient) Hopf subalgebra of $H$
is semisolvable, a similar argument as in Proposition \ref{prop1}
will prove the claim.
\end{rmk}

\subsection{Drinfeld double}\label{sec23}For a semisimple Hopf algebra
$H$, $D(H)=H^{*cop}\bowtie H$ will denote the Drinfeld double of
$H$. $D(H)$ is a Hopf algebra with underlying vector space
$H^{*cop}\otimes H$. The main result in \cite{Etingof3} proves that
if $V$ is an irreducible module of $D(H)$, then the dimension of $V$
divides the dimension of $H$.

Let $^H_H\mathcal{YD}$ denote the category of (left-left)
Yetter-Drinfeld modules over $H$. Objects of this category are
vector spaces $V$ endowed with an $H$-coaction $\rho: V\to H\otimes
V$ and an $H$-action $\cdot: H\otimes V\to V$, which satisfies the
compatibility condition $\rho(h\cdot v) = h_1v_{-1}S(h_3)\otimes
h_2\cdot v_0$, for all $v\in V, h\in H$. Morphisms of this category
are $H$-linear and colinear maps.

Majid first proved that the Yetter-Drinfeld category $_H
^H\mathcal{YD}$ can be identified with the category
$_{D(H)}\mathcal{M}$ of left modules over the quantum double $D(H)$.
see \cite[Proposition 2.1]{majid}.

 More details on $D(H)$ can be
found in \cite[Section 10.3]{Montgomery2}. The following theorem
follows directly from \cite[Proposition 9,10]{Radford2}.

\begin{thm}
Suppose that $H$ is a semisimple Hopf algebra.

(1) The map $G(H^*)\times G(H)\to G(D(H))$, given by
$(\eta,g)\mapsto \eta\bowtie g$, is a group isomorphism.

(2) Every group-like element of $D(H)^*$ is of the form $g\otimes
\eta$, where $g\in G(H)$ and $\eta\in G(H^*)$. Moreover, $g\otimes
\eta\in G(D(H)^*)$ if and only if $\eta\bowtie g$ is in the center
of $D(H)$.
\end{thm}

\begin{cor}\label{cor1}
Suppose that $H$ is a  semisimple Hopf algebra such that $G(D(H)^*)$
is non-trivial. If $gcd(|G(H)|,|G(H^*)|)=1$ then $H$ or $H^*$ has a
non-trivial central group-like element.
\end{cor}
\begin{proof}
Let $1\neq g\otimes \eta\in G(D(H)^*)$. We may assume that $1\neq
g\in G(H)$, since otherwise $\eta\in G(H^*)$ would be a non-trivial
central group-like element, and similarly we may assume that
$\varepsilon\neq\eta\in G(H^*)$. Since $gcd(|G(H)|,|G(H^*)|)=1$, the
order of $g$ and $\eta$ are different. Assume that the order of $g$
is $n$. Then $(g\otimes \eta)^n=g^n\otimes \eta^n=1\otimes
\eta^n\neq1\otimes \varepsilon$ implies that $\eta^n\bowtie 1$ is in
the center of $D(H)$. Hence, $\eta^n$ is a non-trivial central
group-like element in $G(H^*)$. Similarly, we can prove that $G(H)$
also has a non-trivial central group-like element.
\end{proof}

Let $g\in G(H)$, $\eta\in G(H^*)$, and $V_{g,\eta}$ denote the
one-dimensional vector space endowed with the action $h\cdot 1 =
\eta(h)1$, for all $h\in H$, and the coaction $1\mapsto g\otimes1$.

\begin{lem}\label{lem5}\cite[Lemma 1.6.1]{Natale4}The one-dimensional Yetter-Drinfeld modules of $H$
are exactly of the form $V_{g,\eta}$, where $g\in G(H)$ and $\eta\in
G(H^*)$ are such that $(\eta\rightharpoonup h)g = g(h \leftharpoonup
\eta)$ for all $h\in H$, where $\rightharpoonup$ and
$\leftharpoonup$ are the regular actions of $H^*$ on $H$.
\end{lem}

\subsection{Radford biproduct}\label{sec2-3}
Let $A$ be a semisimple Hopf algebra and let ${}^A_A\mathcal{YD}$
denote the braided category of Yetter-Drinfeld modules over $A$. Let
$R$ be a semisimple Yetter-Drinfeld Hopf algebra in
${}^A_A\mathcal{YD}$. Denote by $\rho :R\to A\otimes R$, $\rho
(a)=a_{-1} \otimes a_0 $, and $\cdot :A\otimes R\to R$, the coaction
and action of $A$ on $R$, respectively. We shall use the notation
$\Delta (a)=a^1\otimes a^2$ and $S_R $ for the comultiplication and
the antipode of $R$, respectively.

Since $R$ is in particular a module algebra over $A$, we can form
the smash product (see \cite[Definition 4.1.3]{Montgomery2}). This
is an algebra with underlying vector space $R\otimes A$,
multiplication is given by $$(a\otimes g)(b\otimes h)=a(g_1 \cdot
b)\otimes g_2 h, \mbox{\;for all\;}g,h\in A,a,b\in R,$$ and unit
$1=1_R\otimes1_A$.

Since $R$ is also a comodule coalgebra over $A$, we can dually form
the smash coproduct. This is a coalgebra with underlying vector
space $R\otimes A$, comultiplication is given by $$\Delta (a\otimes
g)=a^1\otimes (a^2)_{-1} g_1 \otimes (a^2)_0 \otimes g_2
,\mbox{\;for all\;}h\in A,a\in R, $$ and counit
$\varepsilon_R\otimes\varepsilon_A$.

As observed by D. E. Radford (see \cite[Theorem 1]{Radford}), the
Yetter-Drinfeld condition assures that $R\otimes A$ becomes a Hopf
algebra with these structures. This Hopf algebra is called the
Radford biproduct of $R$ and $A$. We denote this Hopf algebra by $
R\#A$ and write $a\# g=a\otimes g$ for all $g\in A,a\in R$. Its
antipode is given by
$$S(a\# g)=(1\# S(a_{-1} g))(S_R (a_0 )\# 1),\mbox{\;for
all\;}g\in A,a\in R.$$

A biproduct $R\#A$ as described above is characterized by the
following property(see \cite[Theorem 3]{Radford}): suppose that $H$
is a finite-dimensional Hopf algebra endowed with Hopf algebra maps
$\iota:A\to H$ and $\pi:H\to A$ such that $\pi \iota:A\to A$ is an
isomorphism. Then the subalgebra $R= H^{co\pi}$ has a natural
structure of Yetter-Drinfeld Hopf algebra over $A$ such that the
multiplication map $R\#A\to H$ induces an isomorphism of Hopf
algebras.

Following \cite[Proposition 1.6]{Somm}, if $H\cong R\#A$ is a
biproduct then $H^*\cong R^*\#A^*$ is also a biproduct.

$R$ is called trivial if $R$ is an ordinary Hopf algebra. In
particular, $R$ is an ordinary Hopf algebra if $A$ is normal in $H$,
since $R\cong H/HA^+$ as a coalgebra.

The following lemma is a special case of \cite[Lemma
4.1.9]{Natale4}.
\begin{lem}\label{lem6}
Let $H$ be a semisimple Hopf algebra of dimension $pq^3$, where
$p,q$ are distinct prime numbers. If $p$ divides both $|G(H)|$ and
$|G(H^*)|$, then $H\cong R\#kG$ is a biproduct, where $kG$ is the
group algebra of group $G$ of order $p$, $R$ is a semisimple
Yetter-Drinfeld Hopf algebra in $^{kG}_{kG}\mathcal{YD}$ of
dimension $q^3$.
\end{lem}

\section{Semisimple Hopf algebras of dimension $2q^3$}\label{sec3}
In this section, $H$ will be a non-trivial semisimple Hopf algebra
of dimension $2q^3$, where $q$ is a prime number. Our main aim is to
prove that $H$ is semisolvable. By Proposition \ref{prop1}, it
suffices to prove that $H$ is not simple. When $q=2$, the result
follows from \cite[Theorem 3.5]{Montgomery}. When $q=3$, the result
has been obtained in \cite[Chapter 12]{Natale4}. Therefore, in the
rest of this section, we always assume that $q\geq 5$.

Recall that a semisimple Hopf algebra $A$ is called of Frobenius
type if the dimensions of the simple $A$-modules divide the
dimension of $A$. Kaplansky conjectured that every
finite-dimensional semisimple Hopf algebra is of Frobenius type
\cite[Appendix 2]{Kaplansky}. Dually, the Kaplansky's conjecture
says that the dimensions of the simple $A$-comodules divide the
dimension of $A$. It is still an open problem. However, many
examples show that a positive answer to Kaplansky's conjecture would
be very helpful in the classification of semisimple Hopf algebras.

By \cite[Theorem 1.6]{Etingof2}, $H$ is of Frobenius type.
Therefore, the dimension of a simple $H$-comodule can only be
$1,2,q$ or $2q$. It follows that we have an equation
$$2q^3=|G(H)|+4a+q^2b+4q^2c,\eqno(3.1)$$ where $a,b,c$ are the numbers of
non-isomorphic simple $H$-comodules of dimension $2,q$ and $2q$,
respectively. By Nichols-Zoeller Theorem \cite{Nichols2}, the order
of $G(H)$ divides ${\rm dim}H$. Before discussing the
semisolvability of $H$, we make some preparations.

\begin{lem}\label{lem7}
The order of $G(H)$ can not be $q$.
\end{lem}
\begin{proof}
Suppose on the contrary that $|G(H)|=q$. Let $\chi$ be an
irreducible character of  degree $2$. By Lemma \ref{lem1} (1) and
the fact that $G[\chi]$ is a subgroup of $G(H)$, we know that
$G[\chi]=\{1\}$ is trivial. It follows that the decomposition of
$\chi\chi^*$  as (2.1) gives rise to a contradiction, since $H$ does
not have irreducible characters of degree $3$. Therefore, $a=0$ and
equation (3.1) is $2q^3=q+q^2b+4q^2c$, which is impossible.
\end{proof}

\begin{lem}\label{lem8}
If $|G(H)|=q^2$ then $a=0$ and $b\neq 0$. If $|G(H)|=q^3$ then $H$
is of type $(1,q^3;q,q)$ as a coalgebra.
\end{lem}
\begin{proof}
If $|G(H)|=q^2$ then a similar argument as in Lemma \ref{lem7} shows
that $a=0$. Therefore, equation (3.1) is $2q^3=q^2+q^2b+4q^2c$.
Obviously, $b\neq0$, otherwise a contradiction will occur.

If $|G(H)|=q^3$ then Lemma \ref{lem1} (2) shows that $a=c=0$.
\end{proof}

\begin{lem}\label{lem9}
If $2$ divides both $|G(H)|$ and $|G(H^*)|$ then

(1)\, $H=R\#kC_2$ is a biproduct.

Further, if $kC_2$ is normal in $H$ then

(2)\, $H$ is self-dual, or

(3)\, $H$ fits into an abelian central extension $$k\to kC_2\to H\to
R\to k,$$ and ${\rm dim}V\leq2$ for all irreducible $H$-comodule
$V$.
\end{lem}
\begin{proof}
The first assertion follows from Lemma \ref{lem6}.

Since $R\cong H/H(kC_2)^+$ as a coalgebra and $kC_2$ is normal in
$H$, we have that $R$ is an ordinary Hopf algebra.

First, if $R$ is a dual group algebra then $R^*\subseteq kG(H^*)$.
It contradicts  Nichols-Zoeller Theorem \cite{Nichols2} since ${\rm
dim}R^*=q^3$ does not divide $|G(H^*)|$.

Second, if $R$ is a group algebra then $R^*\subseteq H^*$ is
commutative and the index $[H^*:R^*]=2$. Then it follows from the
Frobenius Reciprocity \cite[Corollary 3.9]{Andrus} that ${\rm
dim}V\leq2$ for all irreducible $H^*$-module $V$. In this case, $H$
fits into an abelian central extension as above by \cite[Proposition
4.6.1]{Natale4}.

Finally, if $R$ is not trivial then it is self-dual \cite{Masuoka2}.
Hence, $H^*\cong R^*\#k^{C_2}\cong R\#kC_2=H$ by \cite[Proposition
1.6]{Somm} (see also Subsection \ref{sec2-3}).
\end{proof}

\begin{lem}\label{lem10}
If $|G(H)|=q^2$ or $q^3$ and $H^*$ has a Hopf subalgebra $K$ of
dimension $2q^2$ then $H$ fits into an extension
$$k\to k\langle g\rangle\to H\to
K^*\to k,$$ where $g\in G(H)$ is of order $q$. Further, if $K$ is
commutative then the extension is abelian.
\end{lem}

\begin{proof}
Considering the map $\pi:H\to K^*$ obtained by transposing the
inclusion $K\subseteq H^*$, we have that ${\rm dim}H^{co\pi}=q$ by
the discussion in Subsection \ref{sec2-2}. Notice that, in our case,
the dimension of every left coideal of $H$ is $1,q$ or $2q$ by Lemma
\ref{lem8}. Therefore, the left coideals contained in $H^{co\pi}$
are all of dimension $1$. If there exists $1\neq g\in G(H)$ such
that $g\in H^{co\pi}$ then $k\langle g\rangle\subseteq H^{co\pi}$
since $H^{co\pi}$ is an algebra. It follows from Lemma \ref{lem3}
that, as a left coideal of $H$, $H^{co\pi}$ decomposes in the form
$H^{co\pi}=k\langle g\rangle$, where $g\in G(H)$ is of order $q$.
Thus $H^{co\pi}$ is normal in $H$, and hence $H$ fits into an
extension as above. The second assertion is obvious.
\end{proof}

\begin{lem}\label{lem11}
If $|G(H)|=2$ then

(1)\, $H$ is commutative, or

(2)\,$H$ contains a Hopf subalgebra $K\subseteq H$ such that $K\cong
k^F$, where $F$ is a non-abelian group of order $2q^2$. Furthermore,
the dimension of an irreducible $H$-module is at most $q$.
\end{lem}

\begin{proof}
Observe that $a\neq 0$ in this case, since otherwise equation (3.1)
can not hold. It follows that $G(H)\cup X_2$ spans a standard
subalgebra of $R(H^*)$, which corresponds to a non-cocommutative
Hopf subalgebra $K$ of dimension $2+4a$. Then $2+4a=2q,2q^2$ or
$2q^3$ by Nichols-Zoeller Theorem \cite{Nichols2}. Obviously,
$2+4a\neq 2q$ since otherwise equation (3.1) can not hold.

If $2+4a=2q^3$ then $H$ is of type $(1,2;2,\frac{q^3-1}{2})$ as a
coalgebra. Then $H$ is commutative by \cite[Proposition
6.8]{Bichon}.

If $2+4a=2q^2$ then $K$ is of type $(1,2;2,\frac{q^2-1}{2})$ as a
coalgebra. Therefore, $K$ is  commutative also by \cite[Proposition
6.8]{Bichon}.

Finally, the Frobenius Reciprocity \cite[Corollary 3.9]{Andrus}
shows that ${\rm dim}V\leq q$ for all irreducible $H$-module $V$.
\end{proof}

\begin{lem}\label{lem12}
If $H$ contains a Hopf subalgebra $K\subseteq H$ such that $K\cong
k^F$ and $G(H)\subseteq K$, where $F$ is a non-abelian group of
order $2q^2$. Then $G(D(H)^*)$ can not contain elements like
$g\otimes \eta$, where $g\in G(H)$ and $\eta\in G(H^*)$ are both of
order $2$.
\end{lem}
\begin{proof}
Suppose on the contrary that there is $g\otimes \eta\in G(D(H)^*)$
such that $g\in G(H)$ and $\eta\in G(H^*)$ are both of order $2$.
Equivalently, there exists a non-trivial one-dimensional
Yetter-Drinfeld module of $H$ of the form $V_{g,\eta}$. See Lemma
\ref{lem5}. Notice that, in our case, the order of $G(H)$ is $2$ or
$2q$.

Consider the projection $\pi:H\to k^{\langle\eta\rangle}$ obtained
by transposing the inclusion $k\langle\eta\rangle\subseteq H^*$.
Since $K$ is commutative and $G(H)\subseteq K$, $g^{-1}ag=a$ for all
$a\in K^{co\pi|_{K}}$, where $K^{co\pi|_{K}}=K\cap H^{co\pi}$.  By
\cite[Theorem 1.6.4]{Natale4}, $K^{co\pi|_{K}}$ is a Hopf subalgebra
of $K$. On the other hand, ${\rm dim}K^{co\pi|_{K}}=2q^2$ or $q^2$
by \cite[Lemma 1.3.4]{Natale4} since ${\rm dim}\pi(K)=1$ or $2$. If
the first case holds true then $K\subseteq H^{co\pi}$, since
dim$K=2q^2$ and $K^{co\pi|_{K}}=K\cap H^{co\pi}$. But this
contradicts Lemma \ref{lem3} since dim$K$ does not divide
dim$H^{co\pi}$. Hence, ${\rm dim}K^{co\pi|_{K}}=q^2$ and
$K^{co\pi|_{K}}$ is a cocommutative Hopf subalgebra of $H$
\cite{Masuoka1}. It is impossible since the order of $G(H)$ is $2$
or $2q$.
\end{proof}

\begin{rmk}\label{rem2}
By \cite[Proposition 9.9]{Etingof2}, $G(D(H)^*)$ is not trivial.
Therefore, if $|G(H)|=2$ and $|G(H^*)|=2,2q$ or $2q^2$ then Lemma
\ref{lem11} and Lemma \ref{lem12} show that $G(D(H)^*)$ contains an
element $g\otimes \eta$ such that $g$ and $\eta$ are of different
order. Then the proof of Corollary \ref{cor1} imply that $H$ or
$H^*$ must contain a non-trivial central group-like element. Hence,
such Hopf algebra is not simple.
\end{rmk}

\begin{lem}\label{lem01}
If $|G(H)|=2q$ then $G(H)$ is cyclic.
\end{lem}
\begin{proof}
Notice that $X_2\neq \varnothing$ and $q^2$ can not divide $a$.
Indeed, if $X_2=\varnothing$ then equation (3.1) can not hold true,
and if $q^2$ divides $a$ then equation (3.1) is
$$2q^3=2q+4q^2a'+q^2b+4q^2c,$$
for some positive integer $a'$, which reduces to a contradiction
$q(2q-4a'-b-4c)=2$.

In addition, $|G[\chi]|=2$ for all $\chi\in X_2$. Then $G(H)$ is
abelian by \cite[Proposition 1.2.6]{Natale4}. The lemma then follows
from the classification of group of order $2q$ \cite{Burnside}.
\end{proof}

In what follows, we shall prove our main theorem by checking every
possible order of $G(H)$. By \cite[Theorem 1.6]{Etingof2},
$|G(H)|\neq 1$. By Lemma \ref{lem7}, $|G(H)|\neq q$. In addition, if
$|G(H)|= 2q^3$ then $H$ is a group algebra. Therefore, it suffices
to check the possibilities that $|G(H)|=2,q^3,2q^2,q^2$ and $2q$. By
Proposition \ref{prop1}, we shall prove that $H$ is not simple.

We will use two classification results from
\cite{Masuoka0,Masuoka2}: Let $q$ be an odd prime. There are two
classes of nontrivial non-isomorphic semisimple Hopf algebras of
dimension $2q^2$. They are dual to each other and of type
$(1,2q;2,\frac{q(q-1)}{2})$ and $(1,q^2;q,1)$ as algebras,
respectively. All nontrivial non-isomorphic semisimple Hopf algebras
of dimension $q^3$ are self-dual and of the same algebra type
$(1,q^2;q,q-1)$.

\begin{prop}\label{prop2}
If $|G(H)|=2$ then

(1)\, $|G(H^*)|\neq2$.

(2)\, If $|G(H^*)|=2q$ then $H$ or $H^*$ contains a non-trivial
central group-like element.

(3)\, If $|G(H^*)|=q^2$ or $q^3$ then $H$ fits into an abelian
extension $$k\to k^F\to H\to k\langle g\rangle\to k,$$ where $F$ is
a group of order $2q^2$ and $g\in G(H^*)$ is of order $q$.

(4)\,If $|G(H^*)|=2q^2$ then $H$ fits into an abelian extension
$$k\to k^F\to H\to k\langle g\rangle\to k,$$ where $F$ is a group of
order $q^3,2q^2$ or $2q$, and $g\in G(H^*)$ is of order $2,q$ or
$q^2$.
\end{prop}
\begin{proof}
Since the aim of our work is to classify non-trivial semisimple Hopf
algebra, we may assume that $H$ contains a commutative Hopf algebra
$k^F\subseteq H$ of dimension $2q^2$ as in Lemma \ref{lem11} and $H$
is of type $(1,2;2,\frac{q^2-1}{2};q,b;2q,c)$ as a coalgebra, where
$b+c\neq0$.

(1)\,Suppose on the contrary that $|G(H^*)|=2$. By Remark
\ref{rem2}, $H$ or $H^*$ has a non-trivial central group-like
element. We may assume that $H$ contains such an element (In our
case, the duality of Lemma \ref{lem11} shows that $H^*$ contains a
commutative Hopf algebra $K\subseteq H^*$ of dimension $2q^2$ and
$H^*$ is of type $(1,2;2,\frac{q^2-1}{2};q,b;2q,c)$ as a coalgebra,
where $b+c\neq0$. Hence, if $H^*$ contains such element, we can do
similarly.). We then consider the quotient Hopf algebra
$\overline{H}=H/H(kG(H))^+$.

First, $\overline{H}$ is not trivial. Indeed, if $\overline{H}$ is a
group algebra then $(\overline{H})^*\subseteq H^*$ is commutative.
Thus ${\rm dim}V\leq[H^*:(\overline{H})^*]=2$ for all irreducible
$H^*$-module $V$ \cite[Corollary 3.9]{Andrus}. It contradicts the
coalgebra type of $H$. In addition, if $\overline{H}$ is a dual
group algebra then $(\overline{H})^*\subseteq kG(H^*)$. It is also
impossible.

Second, $\overline{H}$ is not one constructed in \cite{Masuoka2}.
Indeed, since the modules of $\overline{H}$ are those modules $V$ of
$H$ such that each $x\in kG(H)$ acts as $\varepsilon(x)id_V$ on $V$,
the number of one-dimensional modules of $\overline{H}$ is at most
$2$, since $H$ has only $2$ non-isomorphic simple modules of
dimension $1$. Comparing the results in \cite{Masuoka2}, we get that
$H$ is not the one constructed in \cite{Masuoka2}. All these facts
imply that $|G(H^*)|\neq 2$.

(2)\, It follows from Remark \ref{rem2}.

(3)\, By Lemma \ref{lem10} and \ref{lem11}, $H^*$ fits into an
abelian extension $$k\to k\langle g\rangle\to H^*\to kF\to k,$$
where $F$ is a group of order $2q^2$ and $g\in G(H^*)$ is of order
$q$. The result then follows after dualizing this extension.

(4)\, If $X_2\neq \varnothing$ then $G(H^*)\cup X_2$ spans a
standard subalgebra of $R(H)$, which corresponds to a quotient Hopf
algebra $L$ of dimension $2q^2+4a$. By Nichols-Zoelle Theorem, ${\rm
dim}L=2q^3$. Therefore, $H$ is of type
$(1,2q^2;2,\frac{q^3-q^2}{2})$ as an algebra. Then $H$ or $H^*$ has
a non-trivial central group-like element \cite[Theorem 6.4]{Bichon}.
By Lemma \ref{lem9}, we may assume that $H$ can not contain such an
element since $H$ is not self-dual and there exists simple comodules
of dimension $>2$. We shall examine every possible order of such an
element $g$ in $G(H^*)$.

If the order of $g$ is $2$ then $H^*=R\#k\langle g\rangle$ is a
biproduct by Lemma \ref{lem6}. It follows from Lemma \ref{lem9} that
$H^*$ fits into an abelian extension $$k\to k\langle g\rangle\to
H^*\to R\to k.$$

If the order of $g$ is $q$ we then consider the extension
$$k\to k\langle g\rangle\to H^*\to H^*/H^*(k\langle g\rangle)^+=\overline{H^*}\to k.$$
Since the number of one-dimensional modules of $\overline{H^*}$ is
at most $2$, $\overline{H^*}$ is trivial by comparing the results in
\cite{Masuoka0}. Moreover, $\overline{H^*}$ is not a dual group
algebra since $|G(H)|=2$. It follows that $\overline{H^*}$ is a
group algebra and the extension above is abelian.

If the order of $g$ is $q^2$ we then consider the extension
$$k\to k\langle g\rangle\to H^*\to H^*/H^*(k\langle g\rangle)^+=\overline{H^*}\to k.$$
The classification of semisimple Hopf algebra of dimension $2q$
shows that  $\overline{H^*}$ is trivial \cite{Etingof}. It is
clearly that $\overline{H^*}$ is not a dual group algebra since
$|G(H)|=2$. Hence, $\overline{H^*}$ is a group algebra and the
extension is abelian.

The result then follows after dualizing all these extensions above.

If $X_2= \varnothing$ we then consider the projection $\pi:H^*\to
kF$ obtained by transposing the inclusion $k^F\subseteq H$. Clearly,
${\rm dim}(H^*)^{co\pi}=q$. The decomposition of $(H^*)^{co\pi}$, as
a coideal of $H^*$, shows that $(H^*)^{co\pi}=k\langle g\rangle$ is
a group algebra, where $g\in G(H^*)$ is of order $q$. Therefore $H$
fits into an abelian extension as described
\end{proof}

\begin{lem}\label{lem13}
If $|G(H)|=2q$ then

(1)\, $H$ is of type $(1,2q;2,\frac{q^3-q}{2})$ as a coalgebra, and
$H$ or $H^*$ contains a non-trivial central group-like element, or

(2)\, $H^*$ contains a normal Hopf subalgebra of dimension $q$, or

(3)\, $H$ contains a Hopf subalgebra $K\subseteq H$ such that
$K\cong k^F$, where $F$ is a non-abelian group of order $2q^2$.
Furthermore, the dimension of an irreducible $H$-module is at most
$q$.
\end{lem}

\begin{proof}
Observe that $X_2\neq \varnothing$ and there is a non-cocommutative
Hopf subalgebra $K$ of dimension $2q+4a$, corresponding to the
standard subalgebra of $R(H^*)$ spanned by $G(H)\cup X_2$.

If $2q+4a=2q^3$ then $H=K$ is of type $(1,2q;2,\frac{q^3-q}{2})$ as
a coalgebra. Hence, $H$ or $H^*$ contains a non-trivial central
group-like element \cite[Theorem 6.4]{Bichon}.

If $2q+4a=2q^2$ then $K$ is of type $(1,2q;2,\frac{q(q-1)}{2})$ as a
coalgebra. By Lemma \ref{lem10}, if $|G(H^*)|=q^2$ or $q^3$ then
$H^*$ contains a normal Hopf subalgebra of dimension $q$. In all
other cases, $H=H^{co\pi}\#kC_2$ is a biproduct by Lemma \ref{lem9},
where $\pi:H\to kC_2$ is a projection. Since $K$ is not contained in
$H^{co\pi}$ by Lemma \ref{lem3}, dim$K^{co\pi|_{K}}={\rm dim}(K\cap
H^{co\pi})\neq2q^2$. Hence, dim$\pi(K)=2$ by \cite[Lemma
1.3.4]{Natale4}. Therefore, $\pi|_K:K\to kC_2$ is a surjection.
Moreover, $kC_2$ is contained in $K$. Therefore,
$K=K^{co\pi|_K}\#kC_2$ is also a biproduct. By \cite[Proposition
1.6]{Somm},
$K^*=(K^{co\pi|_K})^*\#k^{C_2}\cong(K^{co\pi|_K})^*\#kC_2$ as a Hopf
algebra. Furthermore, the description of group-like elements of a
biproduct \cite[2.11]{Radford} shows that the order of $G(K^*)$ is
divisible by $2$. All these facts imply that, comparing the results
in \cite{Masuoka0}, $K$ is trivial and hence commutative.

Finally, the Frobenius Reciprocity \cite[Corollary 3.9]{Andrus}
shows that ${\rm dim}V\leq q$ for all irreducible $H$-module $V$.
\end{proof}

\begin{prop}\label{prop3}
If $|G(H)|=q^3$ then

(1)\, $kG(H)$ is a normal Hopf subalgebra of $H$ and $H^*$ has a
central group-like element of order $2$.

(2)\, The order of $G(H^*)$ can not be $q^2$ and $q^3$.
\end{prop}

\begin{proof}
Since the index $[H:kG(H)]=2$ is the smallest prime number dividing
${\rm dim}H$, the main result in \cite{Kobayashi} shows that $kG(H)$
is a normal Hopf algebra of $H$. Then part (1) follows from Lemma
\ref{lem2}. Part (2) is obvious.
\end{proof}

A semisimple Hopf algebra $A$ is called group theoretical if the
category of finite dimensional $A$-modules Rep$A$ is a group
theoretical fusion category. As an immediate consequence of
Proposition \ref{prop3} and \cite[Theorem 1.3.]{Natale2}, we have
the following result.

\begin{cor}
If $|G(H)|=q^3$ and $G(H)$ is abelian then $H$ is group theoretical.
\end{cor}

Let $q$ be a odd prime number. There are five classes of finite
groups of order $q^3$ up to isomorphism: $C_{q^3}$, $C_q\times
C_{q^2}$, $C_q\times C_q\times C_q$, $(C_q\times C_q)\rtimes C_q$
and $C_{q^2}\rtimes C_q$, where $\times$ denotes direct product and
$\rtimes$ denotes semidirect product. The following consequence
proves that $G(H)$ can not be the first one.
\begin{cor}
If $|G(H)|=q^3$ then $G(H)$ is not cyclic and $H$ has a Hopf
subalgebra of dimension $2q^2$.
\end{cor}
\begin{proof}
The group $G(H)$ acts by left multiplication on the set $X_q$. The
set $X_q$ is a union of orbits which have length $1,q,q^2$ or $q^3$.
Since $|X_q|=q$ and the order of stabilizer $G[\chi]$ is at most
$q^2$ for all $\chi\in X_q$, there is only one orbit which has
length $q$. That is, $G[\chi]$ is of order $q^2$ for all $\chi\in
X_q$.

Let $\chi\in X_q$. It follows from the results in \cite[Section
2]{Masuoka3} that the exponent of $G[\chi]$ divides $deg\chi$.
Hence, if $G[\chi]$ is cyclic then $q^2$ divides $q$, a
contradiction. Thus $G(H)$ is not cyclic.

Since $|X_q|=q$ is odd, there is an irreducible character $\chi$ of
degree $q$ which is self-dual. Hence, $\{\chi\}\cup G[\chi]$ spans a
standard subalgebra of $R(H^*)$, which corresponds to a Hopf
subalgebra of dimension $2q^2$.
\end{proof}

\begin{prop}\label{prop4}
If $|G(H)|=2q^2$ then $H$ is not simple.
\end{prop}
\begin{proof}
If $X_2\neq \varnothing$ then there is a non-cocommutative Hopf
algebra $K$ of dimension $2q^2+4a$, corresponding to the standard
subalgebra of $R(H^*)$ spanned by $G(H)\cup X_2$. By Nichols-Zoelle
Theorem, dim$K=2q^3$. Therefore, $H=K$ and $H$ is of type
$(1,2q^2;2,\frac{q^3-q^2}{2})$ as a coalgebra. Then $H$ or $H^*$
contains a non-trivial central group-like element \cite[Theorem
6.4]{Bichon}.

If $X_2=\varnothing$  we then consider the order of $G(H^*)$. By
Lemma \ref{lem10}, if $|G(H^*)|=q^2$ or $q^3$ then $H$ is not
simple. By Proposition \ref{prop2}, if $|G(H^*)|=2$ then $H$ is not
simple. By Lemma \ref{lem13}, if $|G(H^*)|=2q$ then it suffices to
consider the case that $H^*$ has a Hopf subalgebra $K$ of dimension
$2q^2$. In addition, if $|G(H^*)|=2q^2$ then $H^*$ also has a Hopf
subalgebra $K=kG(H^*)$ of dimension $2q^2$. Thus, we consider the
map $\pi: H\to K^*$ obtained by transposing the inclusion
$K\subseteq H^*$. It follows that we have dim$H^{co\pi}=q$. Since
the dimension of every irreducible left coideal of $H$ is $1,q$ or
$2q$, $H^{co\pi}$ decomposes in the form $H^{co\pi}=k\langle
g\rangle$ by Lemma \ref{lem3}, where $g\in G(H)$ is of order $q$.
Hence, $H^{co\pi}$ is normal Hopf subalgebra of $H$.
\end{proof}

\begin{prop}\label{prop5}
If $|G(H)|=q^2$ then $H$ is not simple.
\end{prop}
\begin{proof}
By Proposition \ref{prop3}, we may assume that $|G(H^*)|\neq q^3$.

If $|G(H^*)|=2$ or $2q^2$ then the proposition follows from
Proposition \ref{prop2} and Proposition \ref{prop4}.

By Lemma \ref{lem13}, if $|G(H^*)|=2q$ then it suffices to consider
the case that $H^*$ has a Hopf subalgebra $K$ of dimension $2q^2$.
Considering the map $\pi:H\to K^*$ obtained by transposing the
inclusion $K\subseteq H^*$, we have ${\rm dim}H^{co\pi}=q$. Notice
that the dimension of every irreducible left coideal of $H$ is $1,q$
or $2q$ by Lemma \ref{lem8}. It follows from Lemma \ref{lem3} that,
as a left coideal of $H$, $H^{co\pi}$ decomposes in the form
$H^{co\pi}=k\langle g\rangle$, where $g\in G(H)$ is of order $q$.
Thus, $H^{co\pi}$ is a normal Hopf subalgebra of $H$.

Finally, we consider the case that $|G(H^*)|=q^2$. Considering the
map $\pi:H\to k^{G(H^*)}$ obtained by transposing the inclusion
$kG(H^*)\subseteq H^*$, we have ${\rm dim}H^{co\pi}=2q$. Therefore,
by Lemma \ref{lem3}, as a left coideal of $H$, $H^{co\pi}$
decomposes in the form $H^{co\pi}=k\langle g\rangle\oplus V$, where
$g\in G(H)$ is of order $q$, and $V$ is an irreducible left coideal
of $H$ of dimension $q$. Since $gV$ and $Vg$ are irreducible left
coideals of $H$ isomorphic to $V$, and $gV, Vg$ are contained in
$H^{co\pi}$, we have $gV=V=Vg$. Then \cite[Corollary 3.5.2]{Natale4}
shows that $k\langle g\rangle$ is a normal Hopf subalgebra of
$k[C]$, where $C$ is the simple subcoalgebra of $H$ containing $V$,
and $k[C]$ is a Hopf subalgebra of $H$ generated by $C$ as an
algebra. Clearly, ${\rm dim}k[C]\geq q+q^2$. Moreover, by
Nichols-Zoeller Theorem \cite{Nichols2}, ${\rm dim}k[C]=2q^3,q^3$ or
$2q^2$. If ${\rm dim}k[C]=2q^3$ then $k[C]=H$ and $k\langle
g\rangle$ is a normal Hopf subalgebra of $H$. If ${\rm dim}k[C]=q^3$
then the result follows from \cite{Kobayashi}. If ${\rm
dim}k[C]=2q^2$ then the result follows from Lemma \ref{lem10}.
\end{proof}

By the main result in \cite{Etingof3}, $D(H)$ is of Frobenius type.
Together with Magid's result recalled in Subsection \ref{sec23}, the
dimension of every simple Yetter-Drinfeld $H$-module divides the
dim$H$.

With respect to the left adjoint action $ad : H\otimes H\to H, (ad
h)(a) = h_1aS(h_2)$ and the left regular coaction $\Delta: H \to
H\otimes H$, $H$ becomes a Yetter-Drinfeld $H$-module.

The Yetter-Drinfeld submodules $V \subseteq H$ are exactly the left
coideals $V$ of $H$ such that $h_1V S(h_2)\subseteq V$ , for all $h
\in H$. Thus, a $1$-dimensional Yetter-Drinfeld submodule of $H$ is
exactly the span of a central group-like element of $H$. In
particular, if $\pi:H\to K$ is a Hopf algebra surjection then
$H^{co\pi}$ is a Yetter-Drinfeld submodule of $H$.

We shall afford two quite different proofs of the following
proposition. The first one was pointed to the authors by professor
S. Natale.
\begin{prop}\label{prop6}
If $|G(H)|=2q$ then $H$ is not simple.
\end{prop}

\begin{proof}[First proof]
By the discussion above, it suffices to consider the case that
$|G(H^*)|=2q$, and $H$ and $H^*$ both contain commutative Hopf
subalgebras of dimension $2q^2$, see Lemma \ref{lem13}. Notice that,
in this case, $H$ is of type $(1,2q;2,\frac{q^2-q}{2};q,2q-2)$ as a
(co)algebra.

Let $k^G\subseteq H$ and $k^F\subseteq H^*$ be the commutative Hopf
subalgebras of dimension $2q^2$ obtained in Lemma \ref{lem13}.
Considering the projection $\pi:H\to kF$ obtained by transposing the
inclusion $k^F\subseteq H^*$, we have ${\rm dim}H^{co\pi}=q$. There
are two possible decompositions of $H^{co\pi}$ as a left coideal of
$H$:

(1)\,$H^{co\pi}=k\langle g\rangle$, where $g\in G(H)$ is of order
$q$.

(2)\,$H^{co\pi}=k1\oplus \sum_iV_i$, where $V_i$ is an irreducible
left coideal of $H$ of dimension $2$.

If the first one holds true then $H$ is not simple, and $H$ also
fits into an abelian extension.

If the second one holds true then we may assume $H^{co \pi}
\subseteq k^G$. Notice that the number of irreducible coideals of
dimension $2$ in $H^{co \pi}$ is $\frac{q-1}{2}$.

Note that $H^{co \pi}$ is a Yetter-Drinfeld submodule of $H$, and
decompose it as a sum of irreducible Yetter-Drinfeld submodules. We
may assume that there is a unique summand (spanned by $1$) of
dimension $1$ (otherwise $H$ would contain a central group-like
element, and we are done).

Now, the dimension of every such irreducible summand $W$ must divide
the dimension of $H$. On the other hand, it must be of the form
$2n$, where $1 \leq n \leq \frac{q-1}{2}$ (this follows after
decomposing $W$, which is a coideal, into a sum of irreducible
coideals). So that $n = 1$, because $n = q^k > \frac{q-1}{2}$ if $k
> 1$.

But this implies that the $V_i$'s are Yetter-Drinfeld submodules of
$H$. In particular, $D(H)$ has irreducible modules of dimension $2$
and thus $G(D(H)^*)$ has an element $g \times \eta$ of order $2$,
since $D(H)$ does not has irreducible modules of dimension $3$. This
contradicts Lemma \ref{lem12}.
\end{proof}

\begin{proof}[Second proof]
If the second one holds true then $H^{co\pi}\subseteq K$ and
$K^{co\pi|_{K}}=H^{co\pi}$, where we write $K=k^G$. Then ${\rm
dim}K={\rm dim}K^{co\pi|_{K}}{\rm dim}\pi(K)$ \cite[Lemma
1.3.4]{Natale4} implies that ${\rm dim}\pi(K)=2q$. Consider the
surjection $\pi|_{K}:K\to \pi(K)$. Since $kG(K)\cap
K^{co\pi|_{K}}=k1$, we have that $\pi|_{kG(K)}:kG(K)\to \pi(K)$ is
an isomorphism. Hence $K=K^{co\pi|_{K}}\#kG(K)$ is a biproduct by
the Radford projection theorem \cite[Theorem 3]{Radford}. Since $K$
is commutative, $kG(K)$ is normal in $K$. It follows that
$K^{co\pi|_{K}}$ is an ordinary Hopf algebra, since
$K^{co\pi|_{K}}\cong K/K(kG(K))^+$ as a coalgebra. Since $G(K)$ is
cyclic and $K^{co\pi|_{K}}$ is a group algebra of a cyclic group of
order $q$, we know that $kG(K)$ and $K^{co\pi|_{K}}$ are both
self-dual. Therefore, $K^*\cong (K^{co\pi|_{K}})^*\#k^{G(K)}\cong
K^{co\pi|_{K}}\#kG(K)=K$. But this contradicts the fact that the
finite group $G$ is non-abelian.
\end{proof}

Up to now, we have examined every possible order of $G(H)$ and
proved that $H$ is not simple in all cases. Therefore, together with
Proposition \ref{prop1} and the well-known Burnside's
$p^aq^b$-Theorem, we obtain our main theorem.

\begin{thm}
Let $A$ be a semisimple Hopf algebra of dimension $2q^3$, where $q$
is a prime number. Then $A$ is semisolvable.
\end{thm}

\begin{rmk}
The notion of solvability for semisimple Hopf algebras was
introduced by Etingof et al \cite{Etingof2}. A semisimple Hopf
algebra is called solvable if the category of its finite dimensional
representations is a solvable fusion category. This notion can also
be viewed as a generalization of the notion of solvability for
finite groups. However, the interrelation between solvability and
semisolvability for semisimple Hopf algebras is not clear enough.
For example, Etingof et al  proved \cite[Theorem 1.6]{Etingof2} that
semisimple Hopf algebras of dimension $p^aq^b$ are solvable, while
Galindo and Natale constructed \cite{Galindo} a class of semisimple
Hopf algebras of dimension $p^2q^2$ which is simple as a Hopf
algebra. Our result shows that semisimple Hopf algebras of dimension
$2q^3$ are both solvable and semisolvable.
\end{rmk}

\textbf{Acknowledgments:}\quad The authors are very grateful to
professor Sonia Natale for numerous discussions and very valuable
suggestions about this paper, in particular for providing a new
proof of Proposition \ref{prop6}. This work was partially supported
by the Natural Science Foundation of China (11201231), the China
postdoctoral science foundation (2012M511643), the Jiangsu planned
projects for postdoctoral research funds (1102041C) and Agricultural
Machinery Bureau Foundation of Jiangsu Province (GXZ11003).

\end{document}